\documentclass[draft]{amsart}
\pdfoutput=1
\usepackage{times, enumerate, array, longtable, %
amssymb, diagrams, mathrsfs, bbm, times, txfonts, verbatim}
\usepackage[initials]{amsrefs}

\title{Representations of Symmetric Implication Algebras as Multicubes}
\author{Colin G. Bailey}
\address{School of Mathematics, Statistics \& Operations Research\\
Victoria University of Wellington\\
Wellington, New Zealand\\
}
\email{Colin.Bailey@vuw.ac.nz}
\author{Joseph S.Oliveira}
\address{
Pacific Northwest National Laboratories\\
Richland,  WA\\
U.S.A.}
\email{Joseph.Oliveira@pnl.gov}
\date{2009, February 5}
\subjclass{06A12, 08A30, 08A05}
\keywords{cubes, implication algebras, symmetric implication algebras}
\let\Cal\mathcal
\let\rsf\mathscr

\def\one{{\mathbf 1}}
\def\zero{{\mathbf 0}}

\def\caret{\mathbin{\hat{\hphantom{i}}}}
\def\eqcl[#1]{\pmb{[}#1\pmb{]}}
\def\leftGen{{[\kern-1.1pt[}}
\def\rightGen{{]\kern-1.1pt]}}
\def\s#1{\sigma({{#1}})}
\def\G#1{\Gamma({{#1}})}
\def\os#1{\overline{\sigma({{#1}})}}
\def\env{\operatorname{env}}
\providecommand{\meet}{\mathbin{\wedge}}
\providecommand{\join}{\mathbin{\vee}}
\newcommand{\comp}[1]{\overline{#1}}
\newcommand{\card}[1]{\left| #1\right|}
\ifx\upharpoonright\undefined
     \def\restrict{\hbox{\rm\kern0.166em\accent"12\kern-0.536em$\vert$\kern0.3em}}%
  \else
     \def\restrict{\upharpoonright}%
  \fi
\makeatletter
\def\twoSet#1#2{\left\{%
\vphantom{#2}#1\thinspace\right|\nolinebreak[3]\left.%
  #2%
  \vphantom{#1}%
  \right\}%
}
\def\oneSet#1{\left\lbrace#1\right\rbrace}

\newif\if@nstr
\def\setstrfalse{\let\if@nstr=\iffalse}
\def\setstrtrue{\let\if@nstr=\iftrue}
\def\@nstr #1#2{
\def\@@nstr ##1#1##2##3\@@nstr{\ifx
\@nstr ##2\setstrfalse \else \setstrtrue \fi }
\@@nstr #2#1\@nstr \@@nstr}
\def\@separate#1|#2@{\setFront{#1}\setBack{#2}}
\def\lb#1\rb{\@nstr|{#1} \if@nstr \@separate#1 @ \twoSet{\@setFront}{\@setBack}%
\else \@separate |{#1 }@ \oneSet{\@setBack}\fi%
}
\def\setFront#1{\def\@setFront{#1}}
\def\setBack#1{\def\@setBack{#1}}
\def\Set#1{\lb{#1}\rb}
\newcommand{\bracket}[1]{\left\langle #1\right\rangle}

\def\oneBrk#1{\left\langle#1\right\rangle}
\def\twoBrk#1#2{\left\langle%
\vphantom{#2}#1\thinspace\right|\nolinebreak[3]\left.%
  #2%
  \vphantom{#1}%
  \right\rangle%
}
\def\brk<#1>{\@nstr|{#1} \if@nstr \@separate#1 @ \twoBrk{\@setFront}{\@setBack}%
\else \@separate |{#1 }@ \oneBrk{\@setBack}\fi%
}

\def\lemref#1{\normalfont{lemma}~\ref{#1}}

\def\propref#1{\normalfont{proposition}~\ref{#1}}
\theoremstyle{plain}
\newtheorem{thm}{Theorem}[section]
\newtheorem{lem}[thm]{Lemma}
\newtheorem{cor}[thm]{Corollary}
\newtheorem{prop}[thm]{Proposition}
\newtheorem{defn}[thm]{Definition}

\theoremstyle{remark}
{}
{\newtheorem{example}{Example}[section]}
{}
{\newtheorem{rem}{Remark}[section]}

\@ifpackageloaded{bbm}{\newcommand{\Z}{{\mathbbm{Z}}}

\newcommand{\C}{{\mathbbm{C}}}
}{
\newcommand{\Z}{{\mathbb{Z}}}

\newcommand{\C}{{\mathbb{C}}}}
\makeatother

\begin{document}
    \begin{abstract}
	We show that the 
	variety of symmetric implication algebras is generated from
	cubic implication algebras and Boolean algebras. We do this 
	by developing the notion of a locally symmetric implication 
	algebra that has properties similar to cubic implication 
	algebras and provide a representation of these algebras 
	as subalgebras of a product of a cubic implication algebra 
	and an implication algebra. We then show that every symmetric implication 
	algebra is covered by a locally symmetric implication algebra.
    \end{abstract}
\maketitle
\section{Introduction}
In \cite{SA} the notion of symmetric implication algebra was 
defined and it was shown that cubic lattices in the sense of 
Metropolis and Rota (\cite{MR:cubes}) are closely related to symmetric 
implication algebras. The reflection operator $\Delta$ is shown to be 
a symmetry operator. 

In this paper we consider some aspects of the converse. We first 
define 
a subclass of symmetric implication algebras that satisfy a weak 
version of the Metropolis-Rota axiom. These are shown to arise as 
nice subalgebras of multicubic implication algebras (defined below).
In the last section,  we
give  an envelope construction showing that every 
symmetric implication algebra embeds into a minimal multicubic 
algebra in an extremely nice way.  This shows that the variety of 
symmetric implication algebras is generated by cubic implication 
algebras together with Boolean algebras.

\subsection{Cubic implication algebras}

A cubic implication algebra is an algebraic generalization of the face lattice of 
a finite-dimensional cube based upon work of Metropolis and Rota -- 
see \cites{MR:cubes, BO:eq}. 

We recall some definitions. 
\begin{defn}
    A \emph{cubic implication algebra} is a join semi-lattice with one and a binary 
    operation $\Delta$ satisfying the following axioms:
    \begin{enumerate}[a.]
        \item  if $x\le y$ then $\Delta(y, x)\join x = y$;
        
        \item  if $x\le y\le z$ then $\Delta(z, \Delta(y, x))=\Delta(\Delta(z, 
        y), \Delta(z, x))$;
        
        \item  if $x\le y$ then $\Delta(y, \Delta(y, x))=x$;
        
        \item  if $x\le y\le z$ then $\Delta(z, x)\le \Delta(z, y)$;
        
        \item[] Let $xy=\Delta(1, \Delta(x\join y, y))\join y$ for any $x$, $y$ 
        in $\mathcal L$. Then:
        
        \item  $(xy)y=x\join y$;
        
        \item  $x(yz)=y(xz)$;
    \end{enumerate}
\end{defn} 

It is not difficult to show that this structure with $x\to y= xy$ is 
an implication algebra. 

As can be seen from the axioms, $\Delta(x, y)$ is only of interest if 
$x\geq y$. In the usual examples the most natural way to define 
$\Delta$ produces a partial operation only defined for  $x\geq y$.
The normal way we extend to arbitrary $x$ and $y$ is by letting
$$
\Delta(x, y)=\Delta(x\join y, y).
$$
In \cite{BO:eq} we showed that
\begin{equation}\label{eq:one}
\Delta(x, y)=x\meet \Delta(\one,  xy)
\end{equation}
whenever $x\geq y$.

\begin{defn}
    An \emph{MR-algebra} is a cubic implication algebra satisfying the MR-axiom:\\
    if $a, b<x$ then 
    \begin{gather*}
        \Delta(x, a)\join b<x\text{ iff }a\meet b\text{ does not exist.}
    \end{gather*}
\end{defn}

\begin{example}
    Let $X$ be any set,  and 
    $$
    \rsf S(X)=\Set{\brk<A, B> | A, B\subseteq X\text{ and }A\cap 
    B=\emptyset}.
    $$
    Elements of $\rsf S(X)$ are called \emph{signed subsets} of $X$.
    The operations are defined by 
    \begin{align*}
	1&=\brk<\emptyset,  \emptyset>\\
	\brk<A, B>\join\brk<C, D>&=\brk<A\cap C,  B\cap D>\\
	\Delta(\brk<A, B>, \brk<C, D>)&=\brk<A\cup D\setminus B, 
	B\cup C\setminus A>.
\end{align*}
These are all atomic MR-algebras. The face-poset of an $n$-cube is 
naturally isomorphic to a signed set algebra. 
\end{example}

\begin{example}
    Let $B$ be a Boolean algebra, then the \emph{interval algebra} of $B$ is
$$
\rsf I(B)=\Set{[a, b] | a\le b \text{ in }B}
$$
ordered by inclusion. The operations are defined by
\begin{align*}
	1&=[0, 1]\\
	[a, b]\join[c, d]&=[a\meet c, b\join d]\\
	\Delta([a, b], [c, d])&=[a\join(b\meet\comp d), b\meet(a\join\comp c)].
\end{align*}
These are all atomic MR-algebras. For further details see \cite{BO:eq}.

We note that $\rsf S(X)$ is isomorphic to $\rsf I(\wp(X))$. 
\end{example}

\begin{defn}\label{def:caret}
    Let $\mathcal L$ be a cubic implication algebra. Then for any $x, y\in\mathcal L$ 
    we define the (partial) operation $\caret$ (\emph{caret}) by:
    $$
        x\caret y=x\meet\Delta(x\join y, y)
    $$
    whenever this meet exists. 
\end{defn}

From \eqref{eq:one} we have 
\begin{equation}\label{eq:two}
x\caret y= x\meet\Delta(x\join y, y)= x\meet \Delta(\one, xy).
\end{equation}

\begin{lem}
    If $\mathcal L$ is a cubic implication algebra then 
    $\mathcal L$ is an MR-algebra iff the caret operation is total.
\end{lem}
\begin{proof}
    See \cite{BO:fil} theorem 12.
\end{proof}

Investigating the structure of cubic implication algebras leads us to consider 
certain important subalgebras. 

\begin{defn}
    Let $\mathcal L$ be a cubic implication algebra, and $x\in \mathcal L$. Then
    the \emph{localization of $\mathcal L$ at $x$} is the set
    $$
        \mathcal L_{x}=\Set{y\in\mathcal L | \exists z\geq x\ \Delta(z, 
        x)\le y}.
    $$
\end{defn}

\begin{lem}
    Let $\mathcal L$ be a cubic implication algebra, and $x\in \mathcal L$. Then
    $\mathcal L_{x}$ is the least upwards 
    closed subalgebra of $\mathcal L$ that contains $x$.
\end{lem}
\begin{proof}
    See \cite{BO:eq}. 
\end{proof}

\subsection{Symmetric Implication Algebras}
\begin{defn}\label{def:mcAlg}
	A \emph{symmetric implication algebra} is an implication 
	algebra $\mathcal M$ with a distinguished automorphism $T$ of order two.
\end{defn}

Because of the nature of implication algebras it can be shown that 
if $\tau\colon 2\times 2\to 2\times 2$ is the twist operator, then 
$A_{4}=\brk<2\times 2, \tau>$ and its subalgebras 
$A_{3}=\brk<\Set{{\brk<0,0>}, {\brk<0,1>}, {\brk<1,0>}}, \tau>$ and 
$A_{2}=\brk<\Set{{\brk<0,0>}, {\brk<1,1>}}, \tau>$ are the only 
subdirectly irreducible symmetric implication algebras. 

We let $\mathcal A_{i}$ denote the variety generated by $A_{i}$

\begin{thm}
    \begin{enumerate}[(a)]
        \item  The identity $x\to T(x)=1$ is an equational basis for 
	$\mathcal A_{2}$.
    
        \item  The identity $x\join T(x)=1$ is an equational basis 
	for $\mathcal A_{3}$.
    
        \item  The identity $(x\to T(x))\join y\join T(y)=1$ is an 
	equational basis for $\mathcal A_{2}\join \mathcal A_{3}$.
    \end{enumerate}
\end{thm}

It is easy to see that $\mathcal A_{2}$ is just the variety of 
implication algebras. 

If $\mathcal L$ is a cubic implication algebra then $x\mapsto\Delta(\one,  x)$ 
induces a symmetric implication structure on $\mathcal L$. Because of 
axiom (a) these are in $\mathcal A_{3}$.

We are interested in the variety $\mathcal 
A_{2}\join\mathcal A_{3}$.

We noted above that in cubic implication algebras and MR-algebras the operation
$\brk<x, y>\mapsto x\meet\Delta(\one,  xy)$
is quite important. In general this is not total in symmetric 
implication algebras, but we note that it is defined for $x\geq y$
in each of $\mathcal A_{2}$, $\mathcal A_{3}$ and $\mathcal A_{4}$. 

\begin{defn}
    A \emph{locally symmetric implication algebra} is a symmetric 
    implication algebra in which the meet
    $x\meet T(x\to y)$ exists whenever $x\geq y$.
\end{defn}

Not every symmetric implication algebra is locally symmetric,  but 
all cubic implication algebras are. 

\subsection{Multicubic implication algebras}
We know that every cubic implication algebra can be embedded into the interval 
algebra of a Boolean algebra with upwards-closed image. We take this 
as our starting point for the following definition.

\begin{defn}
    Let $B$ be a Boolean algebra. A \emph{multicubic implication algebra} is any 
    implication algebra isomorphic to an upwards-closed subalgebra of 
    $B\times\rsf I(B)$ -- where we are assuming a partial operation 
    $\Delta$ as usual on $\rsf I(B)$ and extending it as the identity 
    on $B$. 
\end{defn}

Notice that we could have instead defined multicubic implication algebras using
subalgebras of $B_{1}\times\rsf I(B_{2})$ where both $B_{1}$ and 
$B_{2}$ are Boolean algebras. This is immediate as 
$B_{1}\times\rsf I(B_{2})$ is an upwards closed subalgebra of 
$B_{1}\times B_{2}\times\rsf I(B_{1}\times B_{2})$. 

The examples of the next section provide some reason for our 
nomenclature.

\section{Multicubes}
In this section we provide an example of a class of multicubes and a 
construction of a multicubic implication algebra. This example is highly 
geometric, extending one way of looking at $n$-cubes. 
The class of multicubes described here essentially contains all 
finite examples, although the proof of this will not be given in this 
paper.

Let $\Omega=\{ 1,2,\dots,n\}$.  For each $i\in\Omega$ let
$M_i=\{-n_i,\dots,0,\dots,n_i\}$ be a finite $\Z$-module of
odd size.  Let ${\mathcal V}=\prod_{i=1}^nM_i$.  Elements of
$\mathcal V$ will be denoted by $\vec v$ with $i^{\text{th}}$
component $(\vec v)_i$.  We will distinguish certain
formal vectors $e_i$ such that $(e_i)_j=\delta_{ij}$. 
Note also that $M_i$ being of
odd size is equivalent to the property 
$$
\qquad\forall m \ 2m=0\ \Rightarrow m=0
$$ 
This property will be used
implicitly, without mention in many of the results that
follow. 

For $A \subseteq \Omega$ define $X_A={\leftGen}\Set{
e_\alpha | \alpha \in A }{\rightGen}$ where
${\leftGen}S{\rightGen}$ means the submodule generated by
S. 

Clearly
\begin{align*}
\qquad X_A \cap X_B\ &=\ X_{A \cap
B}\\
\text{and }
\qquad X_A + X_B\ &=\ X_{A \cup B}.
\end{align*} 

\subsection{Order and Operations}
Define a partial order {\textbf{P}}=$\langle P,\le \rangle$
by \\ 
$\qquad P = \Set{ \vec a + X_A  |  \vec a \in
{\mathcal V} \text{, and }A \subseteq \Omega}
$ and
$\qquad \le \text{ is } \subseteq$. \\
This is a suborder of the poset of ``affine subspaces'' of
{$\mathcal V$}, and will be called a {\emph{multicube}}.  

\begin{lem}\label{lem:TEST} 
\begin{enumerate}[(a)]
\item
$\vec a +X_A \le \vec b +X_B $ iff $A \subseteq B $ and
$\vec b-\vec a \in X_B$.
\item
$\vec a +X_A < \vec b +X_B $ iff $\vec a +X_A \le \vec b +X_B $
and $A \neq B $.
\item
$\vec a +X_A \le \vec b +X_B $ implies $\vec a +X_B =
\vec b +X_B $.
\end{enumerate}
\end{lem}
\begin{proof} Clear.
\end{proof}

We make \textbf{P} into an upper semilattice by defining some meets 
and all joins as follows:

\begin{defn} 
\begin{align*}
\qquad (\vec a + X_A)\ \wedge\ (\vec b + X_B) &= 
\begin{cases}\vec
c + X_{A \cap B} &\text{ if } c \in (\vec a + X_A)\cap(\vec b +
X_B) \text{ exists}\\
 \text{undefined} &\text{ otherwise}
 \end{cases}
\end{align*}
\end{defn}

Note that $(\vec a + X_A)\wedge(\vec a + X_B)=(\vec a +
X_{A \cap B})$ . 

Before we can define join, we must define an operation that captures 
the difference between two elements of our partial order. 
\begin{defn} ( Of $[-,-]$)\\
If $\vec a,\vec b \in$ {$\mathcal V$} we define \qquad
$$
[\vec a,\vec b] = \Set{i \in \Omega | a_i \neq b_i}
\qquad\text{ where }\vec a = \sum_{i=1}^n a_i e_i. 
$$
\end{defn}
Note that this is not the same as $\leftGen\dots\rightGen$. 

Here are some of the basic properties of $[-,-]$.
\begin{prop} Let $\vec a,\vec b \in \mathcal
V$.Then
\begin{enumerate}[(1).]
\item $[\vec a,\vec b]=[\vec b,\vec a]$;
\item $[\vec a,\vec b]=[\vec a - \vec b,\vec 0]$
and hence $[\vec a + \vec x,\vec b + \vec x]= [\vec a,\vec
b]$;
\item $[\vec a - \vec b,\vec 0]\subseteq [\vec a,\vec 0]
\cup [\vec b,\vec 0]$ with equality iff $[\vec
a,\vec 0] \cap [\vec b,\vec 0]=\emptyset$;
\item $[\vec a,\vec 0]$ is the support of $\vec
a$;
\item $[\vec a,\vec a]=\emptyset$;
\item $[\vec a,\vec 0]\subseteq A\ \Leftrightarrow
\vec a \in X_A$.
\end{enumerate}
\end{prop}
\begin{proof}
Immediate. 
\end{proof}

\begin{defn} ( Of join )
\begin{align*}
(\vec a + X_A)\ \vee\ (\vec b + X_B) &= 
(\vec a + X_{A \cup B \cup [\vec a,\vec b]})
\end{align*}
\end{defn}

The reader may easily verify that \textbf{P} with these
operations is an upper semilattice in which every bounded below pair 
has a greatest lower bound. It also satisfies a
weakened form of distributivity.

\begin{lem}[weak distributivity]
For any $a, b, c$ such that $b\meet c$ exists\ ${a}\join{( b\meet c )} = {(
a\join b )}\meet{( a\join c )}$.
\end{lem}
\begin{proof}
Deferred.
\end{proof}

\subsection{Reflection}

We will now define the partial operation $\Delta :{ P}\times
{ P}\rightarrow { P}$ that corresponds to $\Delta$
in cubic lattices.  This is a particular reflection operator. 
First some
auxiliary notions that help describe the reflection we want to 
abstract. 

\begin{defn}
	A vector $\vec x$ is $B$\emph{-critical} for some set 
	$B\subseteq\Omega$ iff
	$x_{i}\not=0$ implies $i\notin B$.
\end{defn}

\begin{defn}  A vector $\vec x \ \in \vec a +
X_A$ is \emph{critical} if $x_i \neq 0 \Rightarrow e_i \not
\in X_A$ iff $\vec x$ is $A$-critical.
\end{defn} 

\begin{lem} Every ${ a}\in{ P}$ has a unique
critical element, which we will denote by $\Gamma({
a})$. 
\end{lem}
\begin{proof} 
%
	First uniqueness. Suppose that $\vec x_1, \vec x_2$ are two critical 
vectors in $a=\vec a+X_{A}$. Then we have
$[\vec x_1, \vec x_2]\subseteq A$ but also
$[\vec x_1, \vec x_2]\subseteq [\vec x_1, \vec 0]\cup[\vec x_2, \vec 0]
\subseteq\comp A$. Hence $[\vec x_1, \vec x_2]=\emptyset$, i.e. $\vec 
x_1=\vec x_2$. 

For existence, pick any $\vec a\in a$. Define a vector $\vec b$ by
\begin{gather*}
b_i=\begin{cases}
a_i&\text{ if }i\notin A\\
0&\text{ if }i\in A. 
\end{cases}\end{gather*}
Then $[\vec a, \vec b]\subseteq A$ and so $\vec b\in a$,  and
clearly $[\vec b, \vec 0]\subseteq\comp A$. 
\end{proof}

\begin{lem}
	$\mathbf a<\mathbf b$ then $\Gamma(\mathbf b)_i\not=0$ 
implies $\Gamma(\mathbf a)_i=\Gamma(\mathbf b)_i$.
\end{lem}
\begin{proof}
\begin{align*}
\Gamma(\mathbf b)_i\not=0\, &\Rightarrow\, 
i\in\comp{\sigma(\mathbf b)}\\
&\Rightarrow\,i\notin\sigma(\mathbf b)\\
&\Rightarrow\, i\in\comp{[\Gamma(\mathbf a), \Gamma(\mathbf b)]}\\
&\Rightarrow\, \Gamma(\mathbf a)_i=\Gamma(\mathbf b)_i
\end{align*}   
\end{proof}

Coupled with the last definition we have
\begin{defn}
Let $\mathbf a\in\Cal P$. Then we define $\sigma(\mathbf a)$ to 
be the subset of $\Omega$ such that $\mathbf a=\Gamma(\mathbf 
a)+X_{\sigma(\mathbf a)}$
\end{defn}

$\Gamma({ a})$ should be thought of as the center of
the affine subspace ${ a}$.  
Using this representation we can
define $\Delta$.  Note that $\vec x \in  a$ is
critical iff $[\vec x,\vec 0]\subseteq \os a $.

\begin{defn}
If ${ b}\le{ a}$ then 
$$
\Delta({ a},{ b}) = (2\Gamma({a})-\Gamma({ b}))+X_{\sigma({ b})}.
$$ 
\end{defn}

This operation should be thought of as taking the negative
of ${ b}$  relative to ${
a}$.

\begin{prop} \label{prop:basicGD}
Let ${ a},{ b}$ and ${ c}$ be elements of ${
P}$.Then
\begin{enumerate}[(a)]
\item $\Delta({ a},{ a})={ a}$;
\item if $b\le a$ then $\Delta({ a},{ b})\le{ a}$;
\item if $b\le a$ then $\Gamma(\Delta({ a},{ b}))=2\Gamma({
a})-\Gamma({ b})$;
\item if $b\le a$ then $\Delta({ a},\Delta({ a},{ b}))={
b}$;
\item if ${ c}\le{ b}\le{ a }$\ then\ 
$\Delta({ a},{ c})\le\Delta({ a},{
b})$;
\item $\Delta({ a},{ b})={ b}\
\Leftrightarrow \Gamma({ a})=\Gamma({
b})$;
\item if\ ${ b}<{ a}$\ then\ $\Delta({ a},{
b})={ b}\ or\ \Delta({ a},{ b})\ \wedge {
b}=\emptyset$;
\item if\ ${ c}\le{ b}\le{ a }$\ then\
$\Delta({ a},\Delta({ b},{ c}))=\Delta(\Delta({
a},{ b}),\Delta({ a},{ c}))$;
\item if\ ${ a}\ge{ c}$ and ${ b}\ge{ c}$
then
\begin{enumerate}[(a)]
\item $\Gamma(a) + \Gamma(b) = \Gamma({ a } \wedge {
b}) + \Gamma({ a } \vee { b})$;
\item $[\Gamma(a),\Gamma(b)]=[\Gamma({ a } \wedge { b}),
\Gamma({ a } \vee { b})]$;
\item $\Delta({ a},{ c})\ \vee\ \Delta({ b},{
c})= \Delta({ a}\ \vee { b},{ c})\ \vee\
\Delta({ a}\ \wedge { b},{ c})$.
\end{enumerate}
\end{enumerate}
\end{prop}
\begin{proof}
\begin{enumerate}[(1)]
\item $2\vec a-\vec a = \vec a $.
\item ${ b}\le{ a} \Rightarrow \G a -\G b \in X_{\s
a}$.Thus $\G a - (2\G a -\G b) = \G b - \G a \in X_{\s
a}$.
\item Suppose $(2\G a - \G b)_i \neq
0$. Then either
\begin{enumerate}[(i)]
\item
$(\G a)_i = 0$ and $(\G b)_i \neq 0$
or 
\item
$(\G a)_i \neq 0$ and $(\G b)_i =
0$ or  
\item
$(\G a)_i \neq 0$ and $(\G b)_i \neq 0$.
\end{enumerate}
Since $\s b \subseteq \s a$,$\G a - \G b \in X_{\s a}$ and
all of these vectors are critical, case (ii) is
impossible.
Also in case (iii) $(\G a)_i=(\G b)_i$.
Thus we have either case (i) which entails $i \notin \s
b$
 or case (iii) which entails $i \notin \s a$ and hence
$i \notin \s b$.Thus $2\G a - \G b$ is critical.
\item $2\G a -(2\G a - \G b) = \G b$.
\item $(2\G a - \G b) - (2\G a - \G c) = \G c - \G b$ is in 
$X_{\s b}$.

\item $2\G a - \G b = \G b \Leftrightarrow \G
a = \G b$.
\item ${ b}<{ a}$ implies $\s b \subset \s a$ and $\s b
\neq \s a$.There are two cases for $\G b$.
\begin{enumerate}[(i)]
\item
$\G b = \G a$: then
$\Delta({ a},{ b})={ b}$ by (5).
\item
$\G b \neq \G a$: then suppose $\vec x$ is in
both ${ b}$ and $\Delta({ a},{ b})$.Then $\vec x -
\G b \in X_{\s b}$ and $\vec x - (2\G a -\G b) \in X_{\s
b}$.This means that $(\G a - \G b) \in X_{\s b}$.Since
this is a non--zero vector we now have that for some
$i$,either $(\G a)_i \neq 0\ and\ e_i \in X_{\s a}$  or 
$(\G b)_i \neq 0\ and\ e_i \in X_{\s b}$. Both of these
statements contradict the fact that these two vectors are
critical.Hence no such $\vec x$ can exist and so
$\Delta({ a},{ b})\ \wedge {
b}=\emptyset$.
\end{enumerate}
\item $2\G a - (2\G b - \G c)$=$2(2\G a - \G b) - (2\G a -
\G c)$.
\item
\begin{enumerate}[(i)]
\item
Indeed let $\vec c$ be in both ${ a}$ and ${ b}$. Wolog 
$\vec c=\G c$. 
Then $\G a = \sum_{i \notin \s a}c_ie_i$ and $\G b =
\sum_{i \notin \s b}c_ie_i$ and $\Gamma({ a } \wedge { b}) = \sum_{i \notin \s a
\cap \s b}c_ie_i$.

Furthermore $[\G a,\G b] \subseteq \s a \cup \s b$ since 
$\G a - \G b \in X_{\s a}+X_{\s b}$. This means that $\Gamma({ a } \vee {
b})=\sum_{i \notin \s a \cup \s b}c_ie_i$.
Putting all this together means that the left-hand side
of our equation is 
$$2\sum_{i  \in \overline{\s a} \cap \overline{\s
b}}c_ie_i\ +\ \sum_{i  \in (\overline{\s a}\cup\overline{\s
b})\setminus(\overline{\s a}\cap\overline{\s b})}c_ie_i$$
 which is the same as the right-hand side.
\item 
Given $\vec c$ as above
we see that 
\begin{align*}
 \G a &= \sum_{i \in \overline{\s a}}c_ie_i&&(b1)\\
 \G b &= \sum_{i \in \overline{\s b}}c_ie_i&&(b2)\\
 \Gamma({ a } \wedge { b})&= \sum_{i \in \overline{\s a}\cup
\overline{\s b}}c_ie_i&&(b3)\\
\Gamma({ a } \vee { b})&= \sum_{i \in \overline{\s a}\cap
\overline{\s b}}c_ie_i&&(b4)
\end{align*}
This means that 
$[\G a, \G b]= (\os a \setminus \os b)\cup (\os
b \setminus \os a)$ and \\
$[\Gamma({ a } \wedge { b}),\Gamma({ a } \vee { b})]=
((\os a \cup\os b )\setminus(\os
a \cap\os b ))$. It is easy to see
that these are the same.
\item
 Now to compute our $\Delta$ equation.
The left-hand side is $(2\G a - \G c)+X_{\s c} \ \vee\ (2\G
b - \G c)+X_{\s c}$ which equals $(2\G a - \G c)+X_{\s c
\cup [\G a,\G b]}$.

Similarly the right-hand side is $(2\Gamma({ a } \wedge
{ b}) - \G c)+X_{\s c \cup [\Gamma({ a } \wedge {
b}),\Gamma({ a } \vee { b})]}$. 
So it remains to show that \\ 
$\G a -\Gamma({ a } \wedge
{ b}) \in X_{\s c\cup [\G a,\G b]}$. We know that $[\G a ,\G b
]=\os a \setminus \os b  \cup\os b
 \setminus \os a$, and so by
equations (b1) and (b3) we have 
$$
\G a - \Gamma( a\meet  b )=-\sum_{i\in \s a  \cup \os b  }c_ie_i\ 
\in\ X_{[\G a ,\G b]}.
$$
\end{enumerate}
\end{enumerate}
\end{proof}

\def\C(#1,#2){({ #2}\rightarrow{ #1})}
\def\D(#1,#2){\Delta({ #1}, { #2})}
\subsection{Local Complementation}
In the case of the  cubic implication algebras, the delta function enables us to define 
a local complementation that makes the algebra into an implication 
algebra.  However, in the current situation the delta function may 
have fixed points and this implies the complementation function is not 
definable from $\Delta$. However, inspection of the example shows us 
that multicubic implication algebras are
locally complemented with the complement being defined by:
\begin{defn}
	Let $ a\le b$. Then 
	$$
	\C(a, b)=\Gamma(a)+X_{\s a\cup\os b}. 
	$$
\end{defn}

\begin{lem}\label{lem:compGD}
	Let $ a\le b$. Then 
	\begin{enumerate}[(a)]
		\item  $ a\le \C(a, b)$;
	
		\item  $ b\meet \C(a, b)= a$ and $ b\join \C(a, b)= 1$;
	
		\item  if $ b\le c$ then $\C(a, c)\le\C(a, b)$;
	
		\item  $\C(a, {\C(a, b)})= b$;
	
		\item  if $ a\le c$ then $\C(a, {b\meet c})=\C(a, b)\join\C(a, c)$ and 
		$\C(a, {b\join c})=\C(a, b)\meet\C(a, c)$;
	
		\item  $\bracket{[ a,  1], \join, \meet, \C(a, \bullet), 
		 a,  1}$ is a Boolean algebra;
	
		\item  $\bracket{P, \join, \C(\ ,\ )}$ is an implication algebra. 
	\end{enumerate}
\end{lem}
\begin{proof}
	Straightforward. 
\end{proof}

The important relation seen in \cite{MR:cubes} between delta and complementation still exists 
in this case as we have
\begin{lem}
	Let $ a\le b$. Then
	\begin{enumerate}[(a)]
		\item  $\C({\relax\D(b,a)},b)=\Delta( 1, \C(a, 
		b))$;
	
		\item  $\Delta( b,  a)= b\meet \Delta( 1, {\C(a, 
		b)})$. 
	\end{enumerate}
\end{lem}
\begin{proof}
	The second assertion will be proven as part of the first. 
	
	$\Delta( 1, {\C(a, 
		b)})=-\G a + X_{\s a\cup \os b}$ and 
		$\Delta( b,  a)=2\G b-\G a + X_{\s a}$. 
	
	\begin{enumerate}[-]
		\item  $\Delta( b,  a)\le \Delta( 1, {\C(a, 
		b)})$ as 
		\begin{align*}
			[2\G b-\G a,  -\G a]&=[2\G b, 0]\\
			&=[\G b, 0]\\
			&\subseteq\os b \subseteq\s a\cup\os b. 
		\end{align*}
	
		\item  
		\begin{align*}
			\Delta( 1, {\C(a, b)})\meet b&=2\G b-\G a+X_{\s b\cap(\s 
			a\cup\os b)}\\
			&=2\G b-\G a+X_{\s a}\\
			&=\Delta( b,  a). 
		\end{align*}
	
		\item  
		\begin{align*}
			\Delta( 1, {\C(a, b)})\join b&=\G a+X_{\s b\cup(\s 
			a\cup\os b)}\\
			&=\G a+X_{\Omega}\\
			&= 1. 
		\end{align*}
	\end{enumerate}
	The result is now immediate. 
\end{proof}

We will now describe some of 
the algebraic properties of this structure. 
We note that as $\Delta$ has fixed points, this class of 
examples is not a cubic implication algebra, but a careful analysis of fixed 
points reveals a decomposition into a cubic implication algebra and a Boolean 
algebra. In general the cubic implication algebras so obtained are not MR-algebras.
This decomposition corresponds to thinking of a multicube as a family 
of nested cubes. 

\begin{defn}
	$u$ is \emph{somewhere invariant} iff there is some $v\geq u$ 
	with $\D(v, u)=u$. If $u$ is not somewhere invariant then we say 
	it is \emph{nowhere invariant}. 
\end{defn}
\def\si{somewhere invariant}
\def\ni{nowhere invariant}

\begin{lem}
	$u$ is \si\ iff $\comp{[\Gamma(u), \zero]}\cap{\os 
	u}\not=\emptyset$.
\end{lem}
\begin{proof}
	We want to show that there is some $v>u$ with
	$\Delta(v, u) =u$ iff there is some $i\notin\s u$ with $\G u_{i}=0$.
	
	Suppose that such an $i$ exists. Then let
	$v=\G u +X_{\s u\cup\Set{i}}$. Clearly $\G v=\G u$ and $u<
	v$ so that $\Delta(v, u)=u$.
	
	Conversely, if $u<v$ and $\Delta(v, u)=u$ then we have
	$\G u=\G v$ and $\s u\subsetneq\s v$ and so if $i\in\s v\setminus\s 
	u$ we have the desired $i$.
\end{proof}

\begin{cor}\label{cor:ni}
	$u$ is \ni\ iff $\os u\subseteq[\G u, \zero]$. 
\end{cor}
\begin{proof}
	Immediate from the theorem. 
\end{proof}

\begin{thm}
	If $u$ is \ni\ and $u\le v$ then $v$ is \ni. 
\end{thm}
\begin{proof}
	We have $\os v\subseteq\os u\subseteq[\G u, \zero]$. So let
	$i\in\os v$. Then we have $\G v_{i}=\G u_{i}\not=0$ and so
	$i\in[\G v, \zero]$. 
\end{proof}

\begin{defn}
	Let $\mathcal N(P)=\Set{p\in P | p\text{ is \ni}}$. 
\end{defn}
\def\NP{\mathcal N(P)}

\begin{thm}
	$\NP$ is a cubic implication algebra.
\end{thm}
\begin{proof}
	Provided we show that $\NP$ is closed under $\Delta$ the rest 
	follows from 
	\propref{prop:basicGD} and \lemref{lem:compGD}.
	
	If $u\le v$ are both in $\NP$ then
	$\s{\D( v,  u)}=\s u$ and $\G{\D( v,  u)}=2\G v - \G u$ so that
	suppose $i\in\comp{[\Gamma(\D( v,  u)), \zero]}\cap{\os 
	u}$. As $(2\G v-\G u)_{i}=0$ and we know that $(\G v)_{i}$ is either 
	zero or $(\G u)_{i}$ we see that $(\G v)_{i}=(\G u)_{i}=0$. 
	But then $i\in\comp{[\Gamma(u), \zero]}\cap{\os 
	u}$ contradicting the fact that $u$ is \ni.
\end{proof}

The interval $S=[\mathbf{0}+X_{\emptyset}, \mathbf{0}+X_{\Omega}]$ is a 
Boolean algebra of somewhere invariants. It is not hard to see that 
$P$ embeds as an upwards-closed subalgebra of $S\times \NP$ via the 
mapping 
$$
a+X_{A}\mapsto \brk<\mathbf{0}+X_{A\cup[\G a,\mathbf{0}]}, \G 
a+X_{\comp{[\G a, \mathbf{0}]}}>.
$$
This mapping will be explained in a more general case later.

Since $\NP$ is a cubic implication algebra, it embeds as an upwards-closed 
subalgebra of $\rsf I(B)$ for some Boolean algebra $B$, and so we see 
that $P$ is a multicubic implication algebra.

\subsection{Multicubes as Symmetric Implication Algebras}

In the example above we have a symmetric implication algebra with 
$$
T(\mathbf{b})=-\Gamma({ b})+X_{\sigma({ b})}.
$$

It is easy to see that $T^{2}=\text{id}$. We also have 
$$
(\mathbf{a}\to T(\mathbf{a}))\join \mathbf{b}\join T(\mathbf{b})=\one
$$
since 
\begin{align*}
    a\to T(a) &= (a\join T(a))\to T(a)\\
    a\join T(a) &= \G a+ X_{\s cup \s{T(a)}\cup[\G a, -\G a]}\\
    &= \G a + X_{\s a\cup\Set{i| \G a_{i}\not=0}}\\
    (a\join T(a))\to T(a) &= -\G a +X_{\s a\cup(\os a\cap\Set{i|\G 
    a_{i}=0})}\\
    &= -\G a+ X_{\Set{i|\G a_{i}=0}}\\
    &= -\G a+ X_{\comp{[\G a,\mathbf{0}]}}\\
    b\join T(b) &= \G b\restrict\Set{i|\G b_{i}=-\G b_{i}}+X_{\s b\cup\Set{i|\G b_{i}\not=0}}\\
    &= \mathbf{0}+X_{\s b\cup\Set{i|\G b_{i}\not=0}}\\
    \intertext{Therefore}
    (\mathbf{a}\to T(\mathbf{a}))\join \mathbf{b}\join T(\mathbf{b})&=
    c+X_{\comp{[\G a,\mathbf{0}]}\cup \s b\cup\Set{i|\G 
    b_{i}\not=0}\cup [\G a, \mathbf{0}]}\\
    &= \one.
\end{align*}

Furthermore we have 
$$
b\meet T(b\to a)
$$
exists whenever $b\geq a$.
\begin{align*}
    b\meet T(b\to a) & = (\G b+X_{\s b})\meet T(\G a+ X_{\s a\cup\os b})  \\
     & = (\G b+X_{\s b})\meet (-\G a+ X_{\s a\cup\os b})  
\end{align*}
exists iff 
$$
\G b\restrict\os b= -\G a\restrict (\os a\cap\s b)
$$
which is clearly true as the support sets are disjoint.

From this we see that we have a local structure in the variety 
$\mathcal A_{2}\join \mathcal A_{3}$.

Our aim is to show that every local structure in this variety can be 
described as a multicubic implication algebra and we start by describing the 
reflection operator defined by the symmetry $T$. Using this operator 
we then go on and recover a cubic component and a Boolean component of 
the structure and explain the connection between these two components.

\section{Delta Operators}

\begin{defn}\label{def:delta}
	A \emph{Delta-operator} on a multicubic implication algebra is a partial binary 
	function $\Delta$ such that
	\begin{enumerate}[(a)]
		\item $\Delta(b, a)$ is defined and less than $b$ whenever $a\le b$; 
		
		\item $\Delta(a, a)=a$; 
	
		\item if $a\le b$ then $\Delta(b, \Delta(b, a))=a$; 
	
		\item if $a\le b\le c$ then $\Delta(c, a)\le\Delta(c, b)$; 
	
		\item if $a\le b\le c$ then $\Delta(c, \Delta(b, a))=\Delta(\Delta(c, 
		b), \Delta(c, a))$; 
	
		\item if $a<b$ then either $\Delta(b, a)=a$ or $\Delta(b, a)$ and $a$ 
		have no lower bound; 
	
		\item if $a\le b$ then $\Delta(b, a)=b\meet\Delta(\one, b\to a)$. 
	\end{enumerate}
\end{defn}

We note that the $\Delta$ function defined on multicubes is such an 
operator. In the class of local structures in the variety $\mathcal 
A_{2}\join \mathcal A_{3}$ the symmetry $T$ allows the definition of a
Delta operator which then defines a cubic implication algebra 
contained in the multicubic implication algebra. 

\begin{defn}\label{def:deltaX}
    Let $\rsf M$ be a locally symmetric element of the variety $\mathcal 
A_{2}\join \mathcal A_{3}$ with involution $T$,

	We define the 
	partial operation $\Delta$ by:
	$$
		\Delta(b, a)=b\meet T(b\to a)\text{ whenever 
		}a\le b. 
	$$
\end{defn}

We will prove that $\Delta$ is a Delta-operator for every $x$. 

\begin{lem}\label{lem:one}
	$$
	T(a)=\Delta(\one, a). 
	$$
\end{lem}
\begin{proof}
	\begin{align*}
		\Delta(\one, a) & =\one\meet T(\one\to a)  \\
		 & =\one\meet T(a)= T(a). 
	\end{align*}
\end{proof}

\begin{cor}\label{cor:one}
	$$
	\Delta(b, a)=b\meet\Delta(\one, b\to a). 
	$$
\end{cor}
\begin{proof}
	Immediate from the definition and the lemma. 
\end{proof}

\begin{lem}\label{lem:two}
	$$
		\Delta(b, a)\le b. 
	$$
\end{lem}
\begin{proof}
	Trivial from the definition. 
\end{proof}

\begin{lem}\label{lem:three}
	$$
		b\to\Delta(b, a)=T(b\to a). 
	$$
\end{lem}
\begin{proof}
	As 
	\begin{align*}
		b\meet T(b\to a) & =\Delta(b, a).\\
		\intertext{and we also have }
		1 &= T(a)\join a\join((b\to a)\to T(b\to a))\\
		&= a\join((b\to a)\to T(b\to a)) && \text{ as 
		}T(a)\le T(b\to a)\\
		&= ((b\to a)\to a)\join T(b\to a)\\
		&= b\join T(b\to a) && \text{as} (b\to a)\to a=b\join 
		a=b.
	\end{align*}
	Thus
	we see that $T(b\to a)$ is the complement of $b$ over $\Delta(b, 
	a)$ and so equals $b\to\Delta(b, a)$. 
\end{proof}

\begin{lem}\label{lem:four}
	$$
		\Delta(b, \Delta(b, a))=a. 
	$$
\end{lem}
\begin{proof}
	\begin{align*}
		\Delta(b, \Delta(b, a)) & =b\meet T(b\to\Delta(b, a))  \\
		 & =b\meet T(T(b\to a))  \\
		 & =b\meet (b\to a)  &&\text{ as }T^{2}=\text{id}\\
		 & =a. 
	\end{align*}
\end{proof}

\begin{lem}\label{lem:five}
	Let $a\le b\le c$. Then
	$$
		\Delta(c, a)\le \Delta(c, b). 
	$$
\end{lem}
\begin{proof}
	As $a\le b\le c$ we have $c\to b= (c\to a)\join b\geq c\to a$. Thus
	\begin{align*}
		\Delta(c, b) & =c\meet T(c\to b)  \\
		 & \geq c\meet T(c\to a)  \\
		 & =\Delta(c, a). 
	\end{align*}
\end{proof}

\begin{lem}\label{lem:six}
	Let $a\le b$ and $T(a)=a$. Then $T(b)=b$. 
\end{lem}
\begin{proof}
	Suppose that $T(a)=a$. Then we have 
	$T(b\to a)\geq T(a)=a$ and so 
	$\Delta(b, a)= b\meet T(b\to a)\geq a$. Thus
	$a= \Delta(b, \Delta(b, a))\geq \Delta(b, a)$ and so
	$a=\Delta(b, a)$. Now we have 
	\begin{align*}
		a=T(a)=T(\Delta(b, a)) & =T(b\meet T(b\to a))  \\
		 & =T(b)\meet (b\to a)  
	\end{align*}
	and so 
	$T(b)\le (b\to a)\to a=b$. But then we also have 
	$b=T(T(b))\le T(b)$ and so $T(b)=b$. 
\end{proof}

\begin{lem}\label{lem:seven}
	Let $b\geq a$. Then
	$$
		\Delta(b, a)=a\text{ iff }T(b\to a)=b\to a. 
	$$
\end{lem}
\begin{proof}
	We recall that $b\join T(b\to a)=\one$ so that 
	$\Delta(b, a)=b\meet T(b\to a)$ is equal to $a$ iff 
	$T(b\to a)$ is less than the complement of $b$ over $a$, ie 
	$T(b\to a)\le b\to a$. This implies 
	$T(b\to a)= b\to a$. 
\end{proof}

\begin{lem}\label{lem:eight}
	If $a<b$ then either $\Delta(b, a)=a$ or $\Delta(b, a)$ and $a$ 
		have no lower bound.  
\end{lem}
\begin{proof}
	Suppose that $p\le a$ and $p\le\Delta(b, a)$. 
	Then $p\le \Delta(b, a)\le T(b\to a)$ and 
	$T(p)\le T(a)\le T(b\to a)$. Thus 
	$$
		p\join T(p)\le T(b\to a). 
	$$
	But $T(p\join T(p))=T(p)\join p$ and so (by 
	\lemref{lem:six}) we have 
	$T(b\to a)=b\to a$ and hence $\Delta(b, a)=a$. 
\end{proof}

\begin{lem}\label{lem:nine}
	Let $a\le b\le c$. Then 
	$$
		c\to\Delta(b, a)=(c\to b)\meet T(b\to a). 
	$$
\end{lem}
\begin{proof}
	\begin{enumerate}[(a)]
		\item $(c\to b)\meet T(b\to a) \geq b\meet T(b\to a)= 
		\Delta(b, a)$. 
	
		\item 
		\begin{align*}
			c\meet[(c\to b)\meet T(b\to a)] & =[c\meet (c\to b)]\meet T(
			b\to a)  \\
			 & =b\meet T(b\to a)\\
			 &=\Delta(b, a). 
		\end{align*}
	
		\item Working in the distributive lattice $[\Delta(b, a), 
		\one]$ we have 
		\begin{align*}
			c\join[(c\to b)\meet T(b\to a)] & (c\join(c\to b))\meet(c\join 
			T(b\to a))  \\
			 & =\one\meet (c\join T(b\to a))  \\
			 & \geq c\join T(c\to a)  \\
			 & =\one. 
		\end{align*}
	\end{enumerate}
	Thus $c\to\Delta(b, a)$ which is the complement of $c$ over 
	$\Delta(b, a)$ is equal to \\
	$(c\to b)\meet T(b\to a)$. 
\end{proof}

\begin{lem}\label{lem:ten}
	Let $a\le b\le c$. Then 
	$$
		\Delta(c, b)\to\Delta(c, a)=T(b\to a). 
	$$
\end{lem}
\begin{proof}
	\begin{enumerate}[(a)]
		\item $T(b\to a)\geq T(c\to a)\geq \Delta(c, a)$. 
	
		\item 
		\begin{align*}
			T(b\to a)\meet\Delta(c, b) & =T(b\to a)\meet T(c\to 
			b)\meet c  \\
			 & =T((b\to a)\meet(c\to b))\meet c  \\
			 & =T(c\to a)\meet c  \\
			 & =\Delta(c, a). 
		\end{align*}
	
		\item \begin{align*}
			T(b\to a)\join\Delta(c, b) & =T(b\to a)\join[ T(c\to 
						b)\meet c]  \\
			 & =[T(b\to a)\join T(c\to b)]\meet[T(b\to a)\join c]\\
			 &\geq [T(b\to a)\join T(b)]\meet [T(c\to a)\join c]\\
			 &= T((b\to a)\join b)\meet \one\\
			 &=\one. 
		\end{align*}
	\end{enumerate}
	Thus $\Delta(c, b)\to\Delta(c, a)$ which is the complement of 
	$\Delta(c, b)$ over 
	$\Delta(c, a)$ is equal to $T(b\to a)$. 
\end{proof}

\begin{rem}\label{rem:one}
	Note that the lemma shows us that the complement $\Delta(c, 
	b)\to\Delta(c, a)$ is actually independent of $c$. 
\end{rem}

\begin{lem}\label{lem:eleven}
	Let $a\le b\le c$. Then 
	$$
		\Delta(\Delta(c, b), \Delta(c, a))=\Delta(c, 
		\Delta(b, a)). 
	$$
\end{lem}
\begin{proof}
	\begin{align*}
		\Delta(\Delta(c, b), \Delta(c, a)) & =
		\Delta(c, b)\meet T(\Delta(c, b)\to\Delta(c, a))\\
		 & =\Delta(c, b)\meet T(T(b\to a))  \\
		 & =\Delta(c, b)\meet (b\to a)  \\
		 & =c\meet (b\to a)\meet T(c\to b) \\
		 & = c\meet T((c\to b)\meet T(b\to a))\\
		 & = c\meet T(c\to\Delta(b, a))\\
		 &= \Delta(c, \Delta(b, a)). 
	 \end{align*}
\end{proof}

The next step is to show that Delta operators create cubic implication algebras. 
A major part of this proof is analyzing the families of fixed points 
associated with a Delta-operator.

\section{Fixed Points}
Now we turn to a general analysis of the fixed points of a 
Delta-operator. 
Let $\Delta$ be a fixed Delta-operator on a multicubic implication algebra $\mathcal M$. 

We begin by developing several properties of the class of fixed 
points and a variety of characterizations of fixed points. 

\begin{lem}
	Let $ a\le b\le c$. Then 
	\begin{enumerate}[(a)]
		\item  $b\to \Delta(b, a)=\Delta(\one, b\to a)$;
		
		\item  $\D(c, {\C(a, b)\meet  c})= c\meet\D(\one, {\C(a, b)})$;
	
		\item  $\Delta( b,  a)= b\meet \Delta( c, {\C(a, 
		b)\meet  c})$;
			
		\item  if $\Delta( c,  b)= b$ then $\Delta( c, 
		 a)=\Delta( b,  a)$; 
		
		\item  if $\Delta( c,  b)= b$ and 
		$\D(b, a)= a$ then $\Delta( c, 
		 a)= a$. 
	\end{enumerate}
\end{lem}
\begin{proof}
	\begin{enumerate}[(a)]
		\item  
		\begin{align*}
			b\meet\Delta(\one, b\to a)&=\Delta(b, a)\\
			b\meet\Delta(\one, b\to a)&=\one. 
		\end{align*}
	
	
		\item  We note that 
		\begin{align*}
		\C({\relax[\C(a,b)\meet  c]}, c)&=(\C(a, b)\meet  c)\join\C(a, c)\\
		&=(\C(a, b)\join\C(a, c))\meet( c\join\C(a, c))\qquad\text{ in 
		}[ a,  1]\\
		&=\C(a, b). \\
		\text{Thus }\qquad \D(c, {\C(a, b)\meet  c})&=
		 c\meet\D(1, {\C({[{\C(a, b)}\meet c]}, c)})\\
		&= c\meet\D(1, {\C(a, b)}). 
		\end{align*}
	
		\item  
		\begin{align*}
			\D(b, a)&= b\meet\D(1, {\C(a, b)})\\
			&=( b\meet c)\meet\D(1, {\C(a, b)})\\
			&= b\meet\D(c, {\C(a, b)\meet c}). 
		\end{align*}
		
		\item  
		\begin{align*}
			 b=\D(c, b)&\geq\D(c, a)\\
			\D(c, {\C(a, b)\meet c})&\geq\D(c, a)\\
			\text{ Hence }\D(b, a)&= b\meet\D(c,{\relax\C(a, b)\meet c})\\
			&\geq \D(c, a)\\
			\text{ and so }\D(c, {\D(b, a)})&\geq\D(c, {\D(c, a)})= a. \\
			\text{ Also }\D(c, {\D(b, a)})=\D( {\relax\D(c, b)}, {\D(c, a)})&=\D(b, {\D(c,a)})\geq  a\\
			\text{ implies }\D(c, a)&\geq\D(b, a)
		\end{align*}
		which gives the result. 
		
		\item  This is immediate as $\D(c, a)=\D(b, a)= a$. 
	\end{enumerate}
\end{proof}

The converse to the last result is also true but requires a bit more 
work to establish. 

\begin{lem}
	If $ b\geq  a$ then 
	$$
	\D(b, a)= a\iff \D(b, a)\le a 
	$$
	and 
	$$
	\D(b, a)= a\iff \D(b, a)\geq a 
	$$
\end{lem}
\begin{proof}
	The right to left implication is the only nontrivial one. 
	
	If $\D(b, a)\le a$ then $ a=\D(b, {\D(b, a)})\le\D(b, a)$. 
	
	If $\D(b, a)\geq a$ then $ a=\D(b, {\D(b, a)})\geq\D(b, a)$.
\end{proof}

\begin{lem}
	If $ c\geq b\geq  a$ then 
	$$
	\D(b, a)= a\iff \D(1, {\C(a, b)})=\C(a, b),  
	$$
	and 
	$$
	\D(b, a)= a\iff \D(c, {\C(a, b)})=\C(a, b)\meet  c.   
	$$
\end{lem}
\begin{proof}
The first is a special case of the second. 

Using complements in $[ a,  c]$ we have
	\begin{align*}
		\D(b, a)= a&\iff  b\meet\D(c, {\C(a, b)\meet c})= a\\
		&\iff \D(c, {\C(a, b)\meet c})\le\C(a, b)\meet c\\
		&\iff \D(c, {\C(a, b)\meet c})=\C(a, b)\meet c. 
	\end{align*}
\end{proof}

\begin{lem}
		If $ a\le b$ and $\D(1, a)= a$ then $\D(1, b)= b$. 
\end{lem}
\begin{proof} 
	$\C(a, b)\geq  a$ implies $\D(1, {\C(a, b)})\geq\D(1, a)= a$. Thus \\
	$\D(b, a)= b\meet\D(1, {\C(a, b)})\geq  b\meet a= a$. 
	This gives $\D(b, a)= a$. 
 	
 	From this we have $\D(1,{\relax\C(a, b)})=\C(a, b)$ and so 
 	\begin{align*}
 		\C(a, b)\meet\D(1, b)&=\D(1,{\relax\C(a, b)})\meet\D(1, b)\\
 		&=\D(1,{\relax\C(a, b)\meet b})\\
 		&=\D(1, a)\\
 		&= a. 
 	\end{align*}
 	This implies $\D(1, b)\le\C(a,{\relax\C(a, b)})= b$ and so $\D(1, b)= b$. 
\end{proof}

\begin{prop}\label{prop:fixing}
		If $ a\le b\le c$ then 	\\
		$\Delta( c,  a)= a$ iff
		$\Delta( c,  b)= b$ and 
		$\D(b, a)= a$
\end{prop}
\begin{proof}
	The right to left direction has been done. 
	
	As $\D(c, a)= a$ we have  $\D(1, {\C(a, c)})=\C(a, c)$. From
	$\C(a, b)\geq\C(a, c)$ we have $\D(1, {\C(a, b)})=\C(a, b)$ whence 
	$\D(b, a)= a$. 
	
	Now $\D(b, a)= a$ implies $ b\meet\D(c, {\C(a, b)\meet 
	c})= a$ and hence $\D(c, {\C(a, b)\meet c})\le \C(a, b)\meet c$. 
	So we have $\D(c, {\C(a, b)\meet c})=\C(a, b)\meet c$. 
	Then 
	\begin{align*}
		\C(a, b)\meet\D(c, b)&=\C(a, b)\meet c\meet\D(c, b)\\
		&=\D(c, {\C(a, b)\meet c})\meet\D(c, b)\\
		&=\D(c, {\C(a, b)\meet c\meet b})\\
		&=\D(c, a)\\
		&= a. 
	\end{align*}
	Thus $\D(c, b)\le\C(a, {\C(a, b)})= b$ and so $\D(c, b)= b$. 
\end{proof}

Now we specialize analysis of the fixed points of $\Delta$ to a 
particular element of the algebra.  
\def\fix#1{\text{Fix}({#1})}
\def\bot#1#2{{#1}\join\D({#2}, {#1})}
\def\dd#1#2{\delta^{{{#2}}}({#1})}
\def\d#1{\delta({#1})}

\begin{defn}
	Let $ u$ be in $M$. Then 
	\begin{enumerate}[(a)]
		\item  $\fix u=\Set{ v\geq u | \D(v, u)= u}$;
	
		\item  $\Phi( u)=\Set{ x\le u | \D(u, x)= x}$. 
	\end{enumerate}
\end{defn}

The first things we show are that $\fix u\cap[ u,  v]$ has 
a greatest element and $\Phi( u)\cap[ u,  v]$ has 
a least element for any $ v\geq u$. 

\begin{prop}\label{prop:bot}
	If $ u\le x\le v$ then 
	$$
	\D(v, x)= x\iff  u\join\D(v, u)\le  x. 
	$$
\end{prop}
\begin{proof}
	As $\D(v, { u\join\D(v, u)})= u\join\D(v, u)$,  if 
	$ u\join\D(v, u)\le  x\le  v$ then $\D(v, x)= x$ 
	by proposition
	\ref{prop:fixing}. 
	
	$\D(v, x)= x$ and $ u\le x$ gives $\D(v, u)\le\D(v, 
	x)= x$ and so	$ u\join\D(v, u)\le  x$. 
\end{proof}

\begin{cor}\label{cor:bot}
	If $ u\le v\le w$ then $ v\join\D(w, v)= v\join\D(w, u)$. 
\end{cor}
\begin{proof}
	$\D(w, u)\le\D(w, v)$ gives us $ v\join\D(w, v)\geq v\join\D(w, u)$.
	Also $ u\join\D(w, u)\le v\join\D(w, u)\le w$ and so
	$\D(w, { v\join\D(w, u)})= v\join\D(w, u)$, but as
	$ v\le v\join\D(w, u)$ we have $\bot v w \le v\join\D(w, u)$. 
\end{proof}

\begin{prop}
	If $ u\le x\le v$ then 
	$$
	\D(x, u)= u\iff  x\le\C(u,{\relax[ u\join\D(v, u)]})\meet 
	 v. 
	$$
\end{prop}
\begin{proof}
\begin{align*}
	\D(x, u)= u& \iff  
	\D(x, u)= x\meet\D(v,{\relax v\meet\C(u, x)})= u \\
	&\iff \D(v,{\relax v\meet\C(u, x)})\le{ v\meet\C(u, x)}\\
	&\iff \D(v,{\relax v\meet\C(u, x)})={ v\meet\C(u, x)}\\ 
	&\iff \bot u v\le { v\meet\C(u, x)}\\
	&\iff  x\le\C(u,{[\bot u v]})\meet v.   
\end{align*}
\end{proof}

\begin{defn}
	If $ u\le v$ then $\delta^{ v}( u)=
	\C(u,{[ u\join\D(v, u)]})\meet  v$ and 
	$\beta^{{ v}}( u)=\bot u v$. 
	
	$\delta( u)=\delta^{{ 1}}( u)$ and 
	$\beta( u)=\beta^{{ 1}}( u)$. 
\end{defn}

\begin{cor}\label{cor:ddd}
	Let $ u\le v$. Then $\dd {\dd u v} v =\dd u v$. 
\end{cor}
\begin{proof}
	$\dd {\dd u v} v \geq\dd u v$ comes for free.
	
	If $\Delta(x, \dd u v)=\dd u v$ then as $\Delta(\dd u v, u)=u$ we 
	also have $\Delta(x, u)=u$, which implies $x\le\dd u v$. 
	Hence $\dd {\dd u v} v\le \dd u v$.
\end{proof}

We wish to show that the function $\delta$ is order-preserving in both
variables. 
\begin{thm}
	Let $ u\le v\le w$. Then
	\begin{enumerate}[(a)]
		\item  $\dd u v\le\dd u w$;
	
		\item  $\dd u v =\dd u w\meet v$. 
	\end{enumerate}
\end{thm}
\begin{proof}
	\begin{enumerate}[(a)]
		\item  This is clear as $ v\le w$ and
		$\dd u v =\max\left\{\fix u\meet[ u,  v]\right\}$. 
	
		\item  As $\D({\relax\dd{u}{w}}, u)= u$ and $ u\le {\dd u w \meet  
		v}\le\dd u w$
		we have $\D({\relax\dd{u}{w} \meet  v}, u)=u$. Then, as $\dd u w\meet 
		 v\le  v$ this implies ${\dd u w \meet  v}\le\dd u v$. 
		
		From the first part we get $\dd u v\le\dd u w\meet v$. 
	\end{enumerate}
\end{proof}

\begin{cor}\label{cor:bbb}
	Let $ u\le v\le w$. Then
	$\bot u v=(\bot u w)\meet  v$. 
\end{cor}
\begin{proof}
	As $\dd u v =\dd u w\meet v$ we have \\
	$\C(u,{(\bot u v)})\meet 
	v=\C(u,{(\bot u w)})\meet  v$. The left hand side is the complement 
	of $\bot u v$ in $[ u,  v]$ and the right hand side is the 
	complement of $(\bot u w)\meet v$ in $[ u,  v]$. Hence 
	$\bot u v=(\bot u w)\meet v$. 
\end{proof}

It follows from this theorem that we  need only concern ourselves with 
$\d u$ in our study of fixed points. 

\begin{lem}
	Let $ u\le v\le w$ and $\D(w, u)\le\D(v, u)$. Then
	$\D(w, v)= v$. 
\end{lem}
\begin{proof}
	This is immediate as
	$\bot u w\le\bot u v\le  v\le w$ and $\bot u w\in\Phi( 
	w)$. 
\end{proof}

\begin{thm}\label{thm:ddd}
	Let $ u\le v$. Then
	$$
	\d u\le\d v. 
	$$
\end{thm}
\begin{proof}
	$ u\le v$ and $\D({\relax\d{v}}, v)= v$ gives us 
	$\D({\relax\d{v}}, 
	u)=\D(v, u)$. Also $\D({\relax\d{u}}, u)= u$ and so 
	$\D(1,{\relax\C(u,{\relax\d{u}})})={\C(u, {\d{u}})}$. Thus we have
	\begin{align*}
		\D({\relax\d u\join\d v}, u)&=(\d u\join\d v)\meet\D(1,{\relax\C(u, {\d u \join\d 
		v})})\\
		&=(\d u\join\d v)\meet\D(1,{\relax\C(u, {\d v})\meet\C(u, {\d u})})\\
		&=(\d u\join\d v)\meet\D(1,{\relax\C(u, {\d v})})\meet\D(1,{\relax\C(u, {\d 
		u})})\\
		&=(\d u\join\d v)\meet\D(1, {\C(u, {\d v})})\meet{\C(u, {\d 
		u})}\\
		&=[(\d u\meet{\C(u, {\d u})})\join(\d v\meet{\C(u,{\d u})})]\meet
		\D(1,{\relax\C(u, {\d v})})\\
		&=\d v\meet{\C(u, {\d u})}\meet
		\D(1,{\relax\C(u, {\d v})})\\
		&=\D({\relax\d v}, u)\meet\C(u, {\relax\d u})\\
		&\le\D({\relax\d v}, u)\\
		&=\D(v, u).  
	\end{align*}
	Hence, by the lemma we have $\D({\relax\d v\join\d u}, v)= v$ which 
	implies $\d v\join\d u\le \d v$ and hence $\d u\le\d v$. 
\end{proof}

\begin{cor}
	Let $ u\le v\le w$ and $\D(v, u)= u$. Then $\dd u w=\dd v w$. 
\end{cor}
\begin{proof}
	From the theorem we have $\dd u w\le\dd v w$. From 
	$\D({\relax\dd v w}, v)= v$ and $\D(v, u)= u$ we have $\D({\relax\dd v w}, 
	u)= u$ and hence $\dd v w\le\dd u w$. 
\end{proof}

\begin{cor}\label{cor:dddd}
	Let $ u\le v\le w$ and $\D(v, u)= u$. Then 
	$$
	(\bot v w)\meet\C(u, {[\bot u w]})= v. 
	$$
\end{cor}
\begin{proof}
	\begin{align*}
			(\bot v w)\meet\C(u, {[\bot u w]})&=(\bot v w)\meet\C(u, {[\bot u 
		w]})\meet w\\
		&=(\bot v w)\meet\C(v, {[\bot v w]})\meet w\\
		&= v. 
	\end{align*}
\end{proof}

\begin{cor}
	Let $ v_1,  v_2\in\fix u\cap[ u,  w]$ and 
	$ v_1\join\D(w, u)= v_2\join\D(w, u)$. Then $ v_1= v_2$. 
\end{cor}
\begin{proof}
	This is immediate from the last corollary. 
\end{proof}

Putting the last corollary and corollary \ref{cor:bot} together, we see that
the function 
\begin{align*}
	\alpha\colon\fix u\cap[ u,  w]&\to\Phi( w)\cap[ u,  w]\\
	\text{ given by }\qquad\qquad\alpha( x)&= x\join\D(w, u)
\end{align*}
is well-defined, order-preserving and one-one. We also want to show that 
it is onto. 

\begin{lem}
	Let $ u\le y\le w$ and $\D(w, y)= y$. Then there is a $ 
	y'$ with
	\begin{enumerate}[(i)]
		\item  $ u\le y'\le w$;
	
		\item  $\D({y'}, u)= u$;
	
		\item  $ y'\join\D(w, y')= y$. 
	\end{enumerate}
\end{lem}
\begin{proof}
	Let $ y'= y\meet\d u$ (which we note is also equal to $ 
	y\meet\dd u w$ and equal to $ y\meet\C(u, {[\bot u w]})$). 
	Then we easily have $ u\le  y'\le\dd u w\le w$ and so $\D(y', 
	u)= u$. 
	Lastly
	\begin{align*}
		 y'\join\D(w, y')&= y'\join\D(w, u)\\
		&= y'\join\bot u w\\
		&=( y\meet\C(u, {[\bot u w]})\join\bot u w\\
		&=( y\join\bot u w)\meet\left(\C(u, {[\bot u w]})\join\bot u 
		w\right)\\
		&\qquad\qquad\qquad\text{ 
		in }[ u,  1]\\
		&= y\join\bot u w\\
		&= y\text{ by proposition \ref{prop:bot}}. 
	\end{align*}
\end{proof}

\begin{cor}\label{cor:alpha}
	For all $ u\le v$ 
	\begin{align*}
	( v\meet\d u)\join\D(v, u)&= v\\
	\intertext{ and }
	(\bot u v)\meet\d u&= u. 
	\end{align*}
\end{cor}
\begin{proof}
	From the proof of the lemma, as $\D(v, v)= v$ and $ v'= v\meet\d u$. 
	The second equation is a special case of corollary \ref{cor:dddd}. 
\end{proof}

\begin{thm}\label{thm:iso}
	Let $ u\le w$. Then the function
	\begin{align*}
	\alpha\colon\fix u\cap[ u,  w]&\to\Phi( w)\cap[ u,  w]\\
	\text{ given by }\qquad\qquad\alpha( x)&= x\join\D(w, u)\\
	\intertext{ is an isomorphism with inverse }
	\alpha^{-1}( y)&= y\meet\d u. 
\end{align*}
\end{thm}
\begin{proof}
	From the lemma and earlier remarks. 
\end{proof}

Let us now move on to global notions associated with fixed points. 
\begin{defn}
	$ u$ is \emph{somewhere invariant} iff there is some $ v> u$ 
	with $\D(v, u)= u$. If $ u$ is not somewhere invariant then we say 
	it is \emph{nowhere invariant}. 
\end{defn}
\def\si{somewhere invariant}
\def\ni{nowhere invariant}

\begin{lem}\label{lem:ni}
	$u$ is \ni\ iff $u\join\Delta(\one, u)=\one$. 
\end{lem}
\begin{proof}
	Clearly $ u$ is \ni\ iff $ u=\d u$. 
	But 
	\begin{align*}
		\d u=u & \iff (u\join\Delta(\one, u))\to u=u  \\
		 & \iff u\join\Delta(\one, u)=\one. 
	 \end{align*}
\end{proof}

\begin{thm}
	If $u$ is \ni\ and $ u\le v$ then $ v$ is \ni. 
\end{thm}
\begin{proof}
	If $u$ is \ni\ then 
	$v\join\Delta(\one, v)\geq u\join\Delta(\one, u)=\one$ and so
	$v$ is \ni. 
\end{proof}

\begin{cor}
	$ v$ is \ni\ iff there is some $ u\le v$ with $\d u\le v$. 
\end{cor}
\begin{proof}
	If $ v$ is \ni, then $ v=\d v$ and we are done. 
	
	If $\d u\le v$ and $ v$ is \si, then $\d u$ is \si, which contradicts
	corollary \ref{cor:ddd}. 
\end{proof}

\begin{defn}
	Let $\mathcal N( M)=\Set{p\in M | p\text{ is \ni}}$. 
\end{defn}
\def\NP{\mathcal N( M)}

We now use some ideas from cubic implication algebras. 
\begin{defn}
	Let $ a\in M$. Then
	\begin{align*}
		\mathcal L_{ a}^*&=\Set{ b |  b\geq\D(y, a)\text{ for some } 
		y\geq a}\\
		\mathcal L_{ a}&=\mathcal L_{ a}^*\cap\NP. 
	\end{align*}
\end{defn}
\def\La{\mathcal L_{ a}}

\begin{thm}
	$$
	\La=\mathcal L_{\d a}= \mathcal L_{\d a}^*. 
	$$
\end{thm}
\begin{proof}
	The last equality is exactly the last corollary. 
	
	Clearly $ a\le\d a$ implies $\mathcal L_{\d a}^*\subseteq\La$. 
	
	Next, if $ x\geq a$ is in $\NP$ then $ x=\d x\geq\d a$. 
	
	Now, from the proof of theorem \ref{thm:ddd} we note that 
	$ y\geq  a$ implies $\D({\relax\d y},  a)=\D( y,  a)$ and also
	we know that $\d y= y\join\d a$. So, if $ x\geq\D(y, a)$ for some 
	$ y\geq  a$ we may assume that $\d y= y\geq\d a$. 
	
	But then $\D({\relax\D(y,{\relax\d a})},{\relax\D(y, a)})=\D(y,{\relax\D({\relax\d a}, a)})= \D(y, 
	a)$ and so $\D(y,{\relax\d a})\le\d{\D(y, a)}$. As we also have 
	$ a'=\D(y, a)\le y$, symmetrically, we have 
	$\D(y,{\relax\d {a'}})\le\d a$ and so $\D(y,{\relax\d a})=\d{\D(y, a)}$. 
	
	Hence, if $ x\geq a'$ we have $ x\geq\d{a'}=\D(y,{\relax\d a})$ and so
	$ x\in\mathcal L_{\d a}^*$. 
\end{proof}

\begin{thm}
	$\NP$ is a cubic implication algebra. 
\end{thm}
\begin{proof}
	We recall the axioms of a cubic implication algebra $\bracket{\mathcal M,  1, \join, \Delta}$:
	\begin{enumerate}[a.]
		\item  $ x\le y$ implies $\D(y, x)\join x= y$;
	
		\item  $ x\le y\le z$ implies $\D(z,{\relax\D(y,x)})
=\D({\relax\D(z,y)},{\relax\D(z,x)})$;
	
		\item  $ x\le y$ implies $\D(y,{\relax\D(y,x)})= x$; 
	
		\item  $ x\le y\le z$ implies $\D(z, x)\le\D(z, y)$. 
	
		\item[] If $ x y=\D(1, {\D( x\join  y, y)})\join y$ then  
	
		\item  $( x y) y= x\join y$;
	
		\item  $ x( y z)= y( x z)$. 
	\end{enumerate}
	
	All that we need to show are (a) and that $\C(y,{\relax( x\join y)})= x y$. The 
	rest follow from previous results. 
	
	(a) follows from theorem \ref{thm:iso}
	as that theorem implies that if $ x\le  y$ then 
	$\d x\join \D({\relax\d y}, x)=\d y$. As $\d x= x$ and $\d y= y$ we get (a). 
	
	Also, from the last theorem, if $ x\le y$ are both in $\NP$ then 
	so is $\D(y, x)$ and so $\D(y, x)\join\D(1,{\relax\D(y, x)})= 1$. 
	Thus 
	\begin{align*}
		\D(1,{\relax\D(y, x)})\join  x\join y&=\D(1,{\relax\D(y, x)})\join y\\
		&\geq\D(1, {\D(y, x)})\join\D(y, x)\\
		&= 1. 
	\end{align*}
	Also $\D(y, x)= y\meet\D(1,{\relax\C(x, y)})\le\D(1,{\relax\C(x, y)})$ and so
	$\D(1,{\relax\D(y, x)})\le\C(x, y)$. Thus
	\begin{align*}
		 x&= x\meet y\\
		&\le [\D(1,{\relax\D(y, x)})\join  x]\meet y\\
		&\le [\C(x, y)\join x]\meet y\\
		&=\C(x, y)\meet y\\
		&= x. 
	\end{align*}
	Thus we must have $\C(x, y)=\D(1,{\relax\D(y, x)})\join x$ for $ x\le  y$ in $\NP$. 
\end{proof}

\begin{cor}
	For any $ a\in M$ $\La$ is an MR-algebra. 
\end{cor}
\begin{proof}
	As $\La=\mathcal L_{\d a}$ and $\d a\in \NP$ which is a cubic implication algebra. 
\end{proof}

We want to obtain a full characterization of $\La^*$ for any $a$. 
First there are some technical lemmas. 
\def\ay{{a_{y}}}
\begin{lem}\label{lem:ay}
	Let $z\geq y\geq a$ and $\ay=\D(y, a)$. Then 
	$$\bot a z=\bot {\ay} z. $$ 
\end{lem}
\begin{proof}
	First note that it suffices to show that $a\le\bot {\ay} z$ as then 
	$\D(z, a)\le\bot {\ay} z$ and so $\bot a z \le\bot {\ay} z$. 
	Symmetrically, as $a=\D(y, {\ay})$, we also have $\bot {\ay} z\le\bot a z$. 
	\begin{align*}
		a&\le a\join\ay\\
		&=\ay\join\D(y, {\ay})\\
		&\le\ay\join\D(z, {\ay})\qquad\text{ by corollary \ref{cor:bbb}}. 
	\end{align*}
\end{proof}

\begin{lem}
	Let $x$, $y$ and $a$ be such that
	\begin{align*}
		\bot a 1&\le x\\
		\text{ and }\qquad\d a\join\C({\relax\D(y,a)},x)&\le y. 
	\end{align*}
	Then 
	$$
	\D(y,{\relax\C({\relax\D(y,a)}, x)})=\C(a, x). 
	$$
\end{lem}
\begin{proof}
	Let $\ay=\D(y, a)$. Note that $\bot a y\le\bot a 1\le x$, and so 
	$\ay\le x$. 
	
	By definition 
	\begin{align*}
		\D(y,{\relax\C({\relax\ay}, x)})&=y\meet\D(1,{\relax\C({\relax\C({\relax\ay}, x)}, y)}). \\
		\text{And }\C({\relax\C({\relax\ay}, x)}, y)&=\C({\relax\ay}, x)\join\C({\relax\ay}, y)\\
		&=\C({\ay},{\relax x\meet y}). \\
		\text{ Thus }\D(y,{\relax\C({\relax\ay}, x)})&=y\meet\D(1,{\relax\C({\relax\ay},{\relax(x\meet y)})}). 
	\end{align*}
	We also note that $y\geq y\meet x\geq\bot a y$ and so
	$\D(y,{\relax y\meet x})={y\meet x}$. This implies 
	$\D({\relax y\meet x}, a)=\D(y, a)$ and so $\D({\relax y\meet x}, 
	{\ay})=a$. 
	
	Now
	\begin{align*}
		x\meet\D(y,{\relax\C({\ay}, x)})&=x\meet y\meet\D(1,{\relax\C({\ay}, {x\meet y})})\\
		&=\D({\relax x\meet y}, {\ay})\qquad\text{by definition}\\
		&=a\qquad\text{ by the last paragraph}. \\
		x\join\D(y,{\relax\C({\ay}, x)})&=x\join[y\meet\D(1,{\relax\C({\ay},{(x\meet y)})})]\\
		&=[x\join y]\meet[x\join\D(1,{\relax\C({\ay}, {(x\meet
		y)})})]\qquad\text{ in }[a, 1]\\
		\text{ Now }x\geq\bot a 1\text{ and }&y\geq\d a\text{ so that }\\
		x\join y&\geq(\bot a 1)\join\d a\\
		&=\bot a 1\join\C(a,{\relax(\bot a 1)})\\
		&=1\text{ and }\\
		x\geq\bot a 1 \text{ implies }&\D(1, x)=x\text{ so that }\\
		x\join\D(1,{\relax\C({\ay},{(x\meet y)})})\\
		&=\D(1, x)\join\D(1, {\C({\ay}, {x\meet y})})\\
		&=\D(1,{\relax x\join\C({\ay},{(x\meet y)})})\\
		&\geq\D(1, {\relax(x\meet y)\join\C({\ay},{(x\meet y)})})\\
		&=\D(1, 1)\\
		&=1. \\
		\text{ Thus }x\join\D(y,{\relax\C({\ay}, x)})&=1. 
	\end{align*}
	It follows that $\C(a, x)=\D(y,{\relax\C({\ay}, x)})$. 
\end{proof}

\begin{cor}
	Let $y_1$ and $y_2$ be greater than $\d a$. Let $
	a_i=\D({y_i}, a)$ and suppose that $x\geq a_1\join a_2$. Then
	$$
	\D(y_1,{\relax\C(a_1, x)\meet\d{a_1}})=\D(y_2,{\relax\C(a_2, x)\meet\d{a_2}}). 
	$$
\end{cor}
\begin{proof}
	We have $\d{a_i}=\C(a_i,{(\bot {a_i} 1)})$ and so
	$\C({a_i}, x)\meet\d{a_i}= \C({a_i}, x)\meet\C(a_i, {(\bot {a_i} 1)})= 
	\C({a_i}, {(x\join\bot {a_i} 1)})$ and we know that $x\join\bot {a_i} 
	1=\bot x 1$. Thus 
	$$
	\D(y_i,{\relax\C(a_i, x)\meet\d{a_i}})=\D(y_i,{\relax\C(a_i, {\bot x 1})}). 
	$$
	
	Now we have 
	\begin{align*}
		\bot a 1=\bot {a_i} 1&\le\bot x 1\\
		{\C(a_i,{(\bot x 1)})}&\le y_i\\
		\d {a}&\le y_i. 
	\end{align*}
	Then we apply the lemma to get $\D(y_i,{\relax\C(a_i, {(\bot x 1)})})=
	\C(a,{(\bot x 1)})$. As this is independent of $i$ we are done. 
\end{proof}

\begin{thm}\label{thm:laa}
	$$\La^*\simeq\La\times[a, \d a]^{\text{r}}$$
	where $[a, \d a]^{\text{r}}$ is the interval $[a, \d a]$ with 
	the reverse ordering. 
\end{thm}
\begin{proof}
	Define 
	\begin{align*}
		\Psi\colon\La^*&\to\La\times[a, \d a]\text{  by}\\
		\Psi(x)&=\bracket{x\join\d{\ay}, \D(y,{\relax\C({\ay}, x)\meet\d{\ay}})}
	\end{align*}
	where $x\geq\D(y, a)=\ay$ and $y\geq\d a$. 
	
	By the last corollary, $\Psi(x)$ is independent of the choice of $y$. 
	Since $x\join\d{\ay}\geq\d{\ay}$ we must have 
	\begin{align*}
		x\join\d{\ay}&=\d{x\join\d{\ay}}\\
		&\geq\d x\\
		&\geq\d{\ay}
	\end{align*}
	so that $x\join\d{\ay}=\d x$ is also independent of the choice of $y$. 
	
	$\D(y, \bullet)\colon[a, \d a]\to[\ay, \d{\ay}]$ is an isomorphism
	so the second component is in $[a, \d a]$. The second component is also
	equal to $\C(a, {(\bot x 1)})$ so clearly order is reversed on the second 
	component and preserved on the first component. 
	
	\begin{description}
	\item[$\Psi$ is one-one] as if $\Psi(x_1)=\Psi(x_2)$ then 
	$\bracket{\d{x_1}, \C(a, {(\bot {x_1} 1)})}= 
	\bracket{\d{x_2}, \C(a, {(\bot {x_2} 1)})}$
	and so $\d{x_1\join x_2}=\d {x_1}$ and 
	$\bot{x_1\join x_2}1=\bot {x_1}1$. Thus we get
	$\Psi(x_1)=\Psi(x_1\join x_2)$. 
	
	We begin by assuming that $x_2\geq x_1\geq\D(y, a)=\ay$ 
	for some $y\geq\d a$. Then we have 
	\begin{align*}
	x_1\join\d{\ay}&=x_2\join\d{\ay}\text{ and }\\
	\D(y,{\relax\C({\ay}, x_1)\meet\d{\ay}})&=\D(y,{\relax\C({\ay}, x_2)\meet\d{\ay}})\\
	\text{ whence }{\C({\ay}, x_1)\meet\d{\ay}}&={\C({\ay}, x_2)\meet\d{\ay}}\text{ and so}\\
	x_1\meet\C({\ay}, {\d{\ay}})&=x_2\meet\C({\ay}, {\d{\ay}}). \\
	\text{ Now }x_1&=(x_1\join\d{\ay})\meet(x_1\meet\C({\ay}, {\d{\ay}}))\\
	&=(x_2\join\d{\ay})\meet(x_2\meet\C({\ay}, {\d{\ay}}))\\
	&=x_2. 
	\end{align*}
	
	From this, in the general case we get $x_1=x_1\join x_2=x_2$. 
	
	\item[$\Psi$  is onto] Let $z\in\La$ (so that $z\geq\D(y, 
	\d a)=\d{\ay}$ for some $y\geq\d a$) and let $w\in[a, \d a]$. 
	Then let
	$$
	s=z\meet\C({\ay}, {\D(y, w)}). 
	$$
	(The meet exists in $[\ay, 1]$). Then
	\begin{align*}
		[z\meet\C({\ay}, {\D(y, w)})]\join\d{\ay}&=
		(z\join\d{\ay})\meet\left(\C({\ay}, {\D(y, 
		w)})\join\d{\ay}\right)\\
		\text{ and }\C({\ay},{\relax\D(y, w)})\join\d{\ay}&\geq\C({\ay},{\relax\D(y, {\relax\d 
		a})})\join\d{\ay}\text{ as }w\geq\d a\\
		&=\d{\ay}\join\C({\ay},{\relax\d{\ay}})\\
		&=1\text{ so that }\\
		s\join\d{\ay}&=z\join\d{\ay}=z.\\ 
		[z\meet\C({\ay},{\relax\D(y, 
		w)})]\join\C({\ay},{\relax\d{\ay}})
		&=(z\join\C({\ay},{\relax\d{\ay}}))\meet(\C({\ay},{\relax\D(y, w)})\join\C({\ay}, {\d{\ay}}))\\
		&=\C({\ay},{\relax\D(y, w)})\join\C({\ay},{\relax\d{\ay}})\text{ as }
		z\geq\d{\ay}\\
		&=\C({\ay},{\relax\D(y, w)\meet\d{\ay}})\\
		&=\C({\ay},{\relax\D(y, w)\meet\D(y, {\d a})})\\
		&=\C({\ay},{\relax\D(y,{\relax w\meet{\d a}})})\\
		&=\C({\ay},{\relax\D(y, w)})\qquad\text{ as }w\le\d a. \\
		\text{ Hence }\D(y,{\relax\C({\ay}, s)\meet\d{\ay}})&=
		\D(y,{\relax\C({\ay},{\relax s\join\C({\ay},{\relax\d{\ay}})})})\\
		&=\D(y,{\relax\C({\ay},{\relax\C({\ay},{\relax\D(y, w)})})})\\
		&=\D(y,{\relax\D(y, w)})\\
		&=w. \\
		\text{ And so }\Psi(s)&=\bracket{z, w}. 
	\end{align*}	
	\end{description}	
\end{proof}

Using earlier results,we can prove a little more. 
\begin{thm}\label{thm:AASI}
There is an embedding 
	$$e\colon\mathcal M\to\NP\times\Phi(1)$$
	whose  image is upwards closed.  
\end{thm}
\begin{proof}
	Define $e(x)=\bracket{\d x, \bot x 1}$. Clearly $e$ preserves 
	order. By corollary \ref{cor:alpha} we have $\d x\meet(\bot x 1)=
	x$ and so $e$ is one-one. 
	
	Also if $\bracket{a,b}\geq\bracket{\d x, \bot x 1}$ then 
	$y=a\meet b$ exists and is greater than $x$. We need to show 
	that $\d y=a$ and $\bot y 1=b$. 
	\begin{align*}
		\C({\relax a\meet b},a)&=\C(x,a)\join(a\meet b)\\
		&=(\C(x,a)\join a)\meet(\C(x,a)\join b)\\
		&=\C(x,a)\join b\\
		&=b\qquad\text{ as }\C(x,a)\le\C(x,{\d x})=\bot x 
		1\le b.\\
		& \text{ Also }b\geq\bot x 1{\text{ implies 
		}}\D(1,b)&=b.\\
		\text{ Thus }\qquad \D(a,{\relax a\meet b})&=a\meet\D(1,{\relax\C({\relax
		a\meet b},a)})\\
		&=a\meet\D(1,b)\\
		&=a\meet b.
	\end{align*}
	Hence $a\le\d{a\meet b}\le\d a=a$.
	
	Also we have 
	\begin{align*}
		\C({\relax a\meet b},\d{a\meet b})&=\bot{a\meet
		b}1,\text{ and }\\
		\C({\relax a\meet b},\d{a\meet b})&=\C({\relax
		a\meet b},a)\\
		&=b
	\end{align*}
	so that $\bot{a\meet b}1=b$.
	
	Hence $e(a\meet b)=\bracket{a,b}$ and so the range 
	of $e$ is upwards closed.
\end{proof}

From this proof it remains to investigate the structure of $\Phi(
1)$.
There is one very simple case for $\Phi(1)$.
\begin{thm}
	If $M=\La$ for some $a$ then $\Phi(1)=[\bot a 1,1]$.
\end{thm}
\begin{proof}
	Let $x\in\Phi(1)$. Then $x\geq\D(y,a)=\ay$ for some $
	y\geq\d a$, and so $x\geq\bot{\ay}1$. But by lemma \ref{lem:ay}
	this gives $x\geq\bot a 1$. 
	
	As $[\bot a 1,1]\subseteq\Phi(1)$ we are done.
\end{proof}	

Recalling that 
$$
\begin{array}{ccc}
	[\bot a 1,1]&\simeq[a,\d a]&\simeq[a,\d a]^{\text{r}}\\
	y&\mapsto y\meet\d a&\mapsto\C(a,{\relax y\meet\d 
	a})\meet\d a=\C(a,y)\meet\d a
\end{array}
$$
we see that the mapping $e$ of the last theorem is really the same as 
the isomorphism of theorem \ref{thm:laa}.
It is rare however to obtain an isomorphism.

In general, all we know is that $\Phi(\one)$ is an implication 
algebra. In the finite case we are also able to compute its dimension 
from the dimension of the multicube and can show that it is uniform 
-- ie the dimension of $[m,\one]$ for $m$ a minimal element of 
$\Phi(\one)$ is constant. 

\section{Envelopes}
Envelopes arise as natural covering structures in the theory of 
implication algebras and in the theory of cubic implication algebras. 

If $\mathcal I$ is implication algebra then there is a meet-closed 
implication algebra $\env(\mathcal I)$ and an embedding $e\colon\mathcal 
I\to\env(\mathcal I)$ such that 
\begin{enumerate}[(a)]
    \item the range of $e$ is upwards-closed and generates $\env(\mathcal 
    I)$ (using $\meet$ and $\to$); 

    \item if $f\colon \mathcal I\to\mathcal M$ is an implication morphism to 
    a meet-closed implication algebra $\mathcal M$ then there is a unique 
    extension to an implication morphism $\widehat f\colon \env(\mathcal 
    I)\to\mathcal M$ such that $\widehat f_{\circ}e= f$.
\end{enumerate}

Trying to apply the ideas that produce this envelope to symmetric 
implication algebras directly will fail -- although $T$ lifts the 
resulting algebra need not be in the variety of the original.

For cubic implication algebras (which are also implication algebras) a different 
form of envelope is appropriate.

If $\mathcal L$ is a cubic implication algebra then there is an MR-algebra 
$\env(\mathcal L)$ and an embedding $e\colon\mathcal 
L\to\env(\mathcal L)$ such that 
\begin{enumerate}[(a)]
    \item the range of $e$ is upwards-closed and generates $\env(\mathcal 
    L)$ (using caret and $\to$); 

    \item if $f\colon \mathcal L\to\mathcal M$ is a cubic morphism to 
    a MR-algebra $\mathcal M$ then there is a unique 
    extension to a cubic morphism $\widehat f\colon \env(\mathcal 
    L)\to\mathcal M$ such that $\widehat f_{\circ}e= f$.
\end{enumerate}

The analogous theorem for symmetric implication algebras is as follows.
\begin{thm}
    If $\mathcal L$ is a symmetric implication algebra then there is a
    locally symmetric implication algebra 
    $\env(\mathcal L)$ and an embedding $e\colon\mathcal 
    L\to\env(\mathcal L)$ such that 
    \begin{enumerate}[(a)]
	\item the range of $e$ is upwards-closed and generates $\env(\mathcal 
	L)$ (using caret and $\to$); 

	\item if $f\colon \mathcal L\to\mathcal M$ is a symmetric implication morphism to 
	a locally symmetric implication algebra $\mathcal M$ then there is a unique 
	extension to a symmetric implication morphism $\widehat 
	f\colon \env(\mathcal 
	L)\to\mathcal M$ such that $\widehat f_{\circ}e= f$.
    \end{enumerate}
\end{thm}
\begin{proof}
    The simplest case is of a finitely presented implication algebra,  
    ie an implication algebra $\mathcal I$ such that 
    $m(\mathcal I)=\Set{a | a\text{ is minimal}}$ is finite and 
    every $x$ in $\mathcal I$ is above some $a$ in $m(\mathcal I)$.
    
    For this case $T\colon m(\mathcal I)\to m(\mathcal I)$ is a finite 
    permutation of order two and so its disjoint cycle representation
    consists entirely of one and two cycles. 
    
    Notice that if $a\in \mathcal I$ and $T(a)=a$ then every $b\geq a$ is 
    also a fixed point of $T$ -- as
    $1=(b\to T(b))\join a\join T(a)= b\to T(b)$ and so $b\le T(b)$. 
    Therefore $T(b)\le T^{2}(b)=b$ and so $b=T(b)$.
    
    Also $a\join T(a)$ is always a fixed point of $T$. 
    
    Let us first consider the case when $\card{m(\mathcal I)}=2$. If 
    $T$ fixes everything we're done,  as $b\geq a$ always implies 
    $b\meet(b\to a)=a$ in an implication algebra.
    
    Otherwise $m(\mathcal I)=\Set{a, b}$ and $b=T(a)$ -- as $a\le T(a)$ 
    implies $a=T(a)$ and so $b=T(b)$. If $b\le T(a)$ then $T(b)\le a$ 
    and so $T(b)=a$ and $b=T(a)$. 
    
    Then $T$ is the identity above $a\join b$ and so we take the 
    structure $\mathcal L=[a\join b, \one]\times \rsf I([a, a\join b])$. Then 
    define 
    $\varphi\colon\mathcal I\to \mathcal L$ by 
    $$
    \varphi(x)=
    \begin{cases}
        \brk<x\join a\join b, [a, x\meet(a\join b)]> & \text{ if }x\geq a  \\
	\brk<x\join a\join b, \Delta(\one, [a, T(x\meet(a\join b))]> & \text{ if }x\geq b
    \end{cases}
    $$
    This is an implication morphism because it is on each interval. 
    
    In general we use induction on $\card{m(\mathcal I)}$. Pick any 
    $a\in m(\mathcal I)$ and consider $m'(\mathcal I)=m(\mathcal 
    I)\setminus\Set{a,  T(a)}$. Then $\mathcal I'=\bigcup_{b\in m'(\mathcal I)}$
    can be embedded into some $M=J\times\rsf I(B)$ with upwards-closed 
    image. Then we consider the set $A=\Set{a\join b | b\in m'}$ and 
    its image $A'$. In both $\mathcal I$ and $M$ $\bigwedge A$ exists and 
    the intervals $[\bigwedge A, \one]$ and $[\bigwedge A', \one]$ 
    are isomorphic -- in fact the embedding can be 
    extended to an isomorphism from $[\bigwedge A, \one]$ to 
    $[\bigwedge A', \one]$. 
    
    If $a=T(a)$ then we extend the embedding in a natural way to 
    include $[a, \bigwedge A]$ to $J\times[a, \bigwedge A]$.
    
    If $a\not= T(a)$ we extend the embedding as in the cardinality 
    $2$ case to get $[a, \bigwedge A]$ going into $\rsf I(B\times [a, 
    \bigwedge A])$ and being $\Delta$ on $[T(a),  T(\bigwedge A)]$.

    For the general case,  we use free algebras as usual -- take the 
    free multicubic implication algebra (with $\Delta$ made total as described 
    above) on generators $\mathcal I$ with extra relations
    $$
    x\join_{\rsf F}y=x\join_{\mathcal I}y; \ 
    x\to_{\rsf F}y=x\to_{\mathcal I}y; \
    \Delta(\one,  x)_{\rsf F}= T(x)
    $$
    for all $x$ and $y$ in $\mathcal I$. 
    
    The mapping is clear. Checking that it is an embedding uses the 
    finite case -- any amount of information we need in a computation 
    is finite and can be checked in a finitely presented subalgebra. 
    Because it must work therein,  it works in this bigger arena. 
    
    Checking that the range is upwards-closed is a simple induction 
    on complexity of words using the extra relations.
\end{proof}

%
%
%
%
%
%


\begin{bibdiv}
\begin{biblist}
    \DefineName{cgb}{Bailey, Colin G.}
    \DefineName{jso}{Oliveira,  Joseph S.}
    
\bib{SA}{article}{
title={Representation of Cubic Lattices by Symmetric 
Implication Algebras}, 
author={Abad, Manuel}, 
author={Varela, Jos\'e Patricio D'az}, 
journal={Order}, 
volume={23}, 
date={2006}, 
pages={173--178}
}

\bib{BO:eq}{article}{
title={An Axiomatization for Cubic Algebras}, 
author={cgb}, 
author={jso}, 
book={
    title={Mathematical Essays in Honor of Gian-Carlo Rota}, 
    editor={Sagan,  Bruce E.}, 
    editor={Stanley, Richard P.}, 
    publisher={Birkha\"user}, 
    date={1998.}, 
}, 
pages={305--334}
}

\bib{BO:fil}{article}{
author={cgb}, 
author={jso}, 
title={Cube-like structures generated by filters}, 
journal={Algebra Universalis}, 
volume={49}, 
date={2003}, 
pages={129--158}
}

\bib{MR:cubes}{article}{
author={Metropolis, Nicholas}, 
author={Rota,  Gian-Carlo}, 
title={Combinatorial Structure of the faces 
of the n-Cube}, 
journal={SIAM J.Appl.Math.}, 
volume={35}, 
date={1978}, 
pages={689--694}
}
\end{biblist}
\end{bibdiv}
\end{document}